\documentclass[reqno]{amsart}

\usepackage{amsmath}
\usepackage{amssymb}
\usepackage{amsthm}
\usepackage{mathrsfs}
\usepackage{mathtools}
\usepackage{graphicx}
\usepackage{longtable}
\usepackage{booktabs}
\usepackage{float}
\usepackage{aliascnt}
\usepackage[font=footnotesize]{caption}
\usepackage{enumerate}
\usepackage{cleveref}
\usepackage{enumitem}

\usepackage[top=1in, bottom=1in, left=1in, right=1in]{geometry}

\usepackage{tikz}
\usetikzlibrary{arrows.meta,calc}

\usepackage{caption}

\newcommand{\FF}{\mathbb{F}}

\newcommand{\ZZ}{\mathbb{Z}}

\renewcommand{\pmod}[1]{\text{ (mod }#1)}

\theoremstyle{definition}
\newtheorem{theorem}{Theorem}[section]

\newtheorem{proposition}[theorem]{Proposition}
\newtheorem{lemma}[theorem]{Lemma}
\newtheorem{corollary}[theorem]{Corollary}

\allowdisplaybreaks[4]

\newcommand{\xline}[2][]{%
  \mathrel{\overset{#2}{\rule[0.5ex]{1.5cm}{0.5pt}}}%
}

\title{Sharp vertex connectivity of Markoff graphs modulo $p$}
\author{Jie Ma}
\email{jiema@ustc.edu.cn}
\address{School of Mathematical Sciences, University of Science and Technology of China, Hefei, 230026, People's Republic of China}
\address{Yau Mathematical Sciences Center, Tsinghua University, Beijing 100084, China}
\author{Mengxi Yang}
\email{yangmx221b@outlook.com}
\address{School of Mathematical Sciences, University of Science and Technology of China, Hefei, 230026, People's Republic of China}
\author{Zichen Yang}
\email{zichenyang.math@gmail.com}
\address{School of Mathematical Sciences, University of Science and Technology of China, Hefei, 230026, People's Republic of China}
\date{\today}

\begin{document}

\begin{abstract}
For a prime $p$ and $k\in\FF_p$, the generalized Markoff graph $G_{p,k}$ is an undirected graph whose vertices are the solutions over the finite field $\FF_p$ of the normalized Markoff equation
\[
x_1^2+x_2^2+x_3^2=x_1x_2x_3+k,
\]
where two vertices are adjacent if they differ by a Vieta involution.
The Markoff graph $G_p$ is obtained from $G_{p,0}$ by removing the origin. The structure of $G_p$ has been the subject of extensive study; in particular, a major breakthrough of Bourgain, Gamburd, and Sarnak established that $G_p$ contains a giant connected component. Combined with Chen's remarkable divisibility theorem, this implies that $G_p$ is connected for all sufficiently large primes $p$. In the same paper, Bourgain, Gamburd, and Sarnak further asked whether the family $\{G_p\colon \text{primes }p\geq 5\}$ forms an expander family. This motivates us to investigate the robustness of connectivity in the Markoff graphs. 

The main result of this paper is proved in the general setting: for every prime $p\geq5$ and every $k\in\FF_p\setminus\{4\}$, each connected component $C$ of $G_{p,k}$ with $|V(C)|\geq 3$ is $2$-connected. Reducing to the case $k=0$, we conclude that if the Markoff graph $G_p$ is connected, then it is in fact $2$-connected. Consequently, the Markoff graph $G_p$ is $2$-connected for all sufficiently large primes $p$. This is sharp in the sense that $G_p$ is not $3$-connected for any prime $p\geq 7$. 
\end{abstract}

\maketitle

\section{Introduction}

The classical \emph{Markoff equation}
\begin{equation}
\label{eq:classical-markoff}
x_1^2+x_2^2+x_3^2=3x_1x_2x_3
\end{equation}
was introduced by Markoff \cite{Markoff1879} in his study of indefinite binary quadratic forms and Diophantine approximation. The Markoff equation and its generalizations arise in diverse areas of mathematics, ranging from Diophantine approximation to differential equations to hyperbolic geometry; for a comprehensive overview of the subject, see the survey by Silverman \cite{Silverman2026}. A fundamental feature of the Markoff equation is that the surface defined by it admits three non-commuting \emph{Vieta involutions}
\[
\begin{aligned}
(x_1,x_2,x_3)&\longmapsto(3x_2x_3-x_1,x_2,x_3),\\
(x_1,x_2,x_3)&\longmapsto(x_1,3x_1x_3-x_2,x_3),\\
(x_1,x_2,x_3)&\longmapsto(x_1,x_2,3x_1x_2-x_3).
\end{aligned}
\]
Markoff \cite{Markoff1879} proved that every positive integral solution to \eqref{eq:classical-markoff} can be obtained by repeatedly applying three Vieta involutions to the fundamental solution $(x_1,x_2,x_3)=(1,1,1)$. From the graph-theoretic perspective, one may form an undirected graph whose vertices are the positive integral solutions of \eqref{eq:classical-markoff}, with two solutions connected by an edge if they differ by a Vieta involution. In this sense, Markoff's theorem asserts that the graph is connected. Indeed, it is an infinite tree, known as the \emph{Markoff tree}.

Baragar \cite{Baragar1991} considered a local analogue of the connectivity theorem. To be precise, let $p\geq 5$ be a prime, and consider the normalized Markoff equation 
\begin{equation}
\label{eq:normalized-markoff}
x_1^2+x_2^2+x_3^2=x_1x_2x_3
\end{equation}
over the finite field $\FF_p$. Note that the scaling
\[
(x_1,x_2,x_3)\longmapsto(3x_1,3x_2,3x_3)
\]
identifies the solutions of \eqref{eq:classical-markoff} with those of \eqref{eq:normalized-markoff}.

Let $X_p$ be the set of solutions $(x_1,x_2,x_3)\neq(0,0,0)$ of the normalized Markoff equation. The associated Vieta involutions on $X_p$ are given by
\begin{equation}
\label{eqn:vieta_involution}
\begin{aligned}
m_1(x_1,x_2,x_3)&\coloneqq(x_2x_3-x_1,x_2,x_3),\\
m_2(x_1,x_2,x_3)&\coloneqq(x_1,x_1x_3-x_2,x_3),\\
m_3(x_1,x_2,x_3)&\coloneqq(x_1,x_2,x_1x_2-x_3).
\end{aligned}
\end{equation}
The \emph{Markoff graph} $G_p$ is an undirected graph whose vertex set is $X_p$ and whose edges connect two vertices $x,y\in X_p$ if $m_i(x)=y$ for some $i\in\{1,2,3\}$. For every prime $p\geq5$, Baragar \cite{Baragar1991} conjectured that the Markoff graph $G_p$ is also connected.

Bourgain, Gamburd, and Sarnak \cite{BGS2016note,BGS2026} made the first major progress toward this conjecture. They proved that, for every $\varepsilon>0$ and all sufficiently large primes $p$ depending on $\varepsilon$, the graph $G_p$ has a giant connected component \(\mathcal{C}_p\) in the sense that
\[
\bigl|X_p\setminus\mathcal{C}_p\bigr|<p^\varepsilon.
\]
This was followed by another breakthrough of Chen \cite{Chen2024}, who established that the cardinality of every connected component of $G_p$ is divisible by $p$. Together, these two results imply that $G_p$ is connected for every sufficiently large prime $p$. Recently, Martin \cite{Martin2025} gave a short alternative proof to Chen's divisibility theorem.

Eddy, Fuchs, Litman, Martin, and Tripeny \cite{EddyEtAl2025} further developed the approach of Bourgain, Gamburd, and Sarnak to show that Baragar's conjecture is true for all primes greater than $3.449\cdot10^{392}$. On the other hand, Brown \cite{Brown2025} developed a fast connectivity test and verified the conjecture for all primes in $[5, 10^6)$. Nevertheless, Baragar's full conjecture remains open. 

Beyond mere connectivity, Bourgain, Gamburd, and Sarnak \cite{BGS2026} raised the stronger question of whether the family of Markoff graphs
$
\{G_p:p\geq5\text{ prime}\}
$
forms an expander family. Important qualitative evidence for this question was provided by de Courcy-Ireland, who proved that $G_p$ is non-planar for every $p>7$ \cite{de2024non}. 
This expansion problem motivates us to study the robustness of connectivity in the Markoff graphs. More precisely, the main goal of this paper is to investigate the $2$-connectivity of the Markoff graphs. Here and throughout, by $2$-connectivity, we mean $2$-vertex-connectivity: the graph has at least three vertices and remains connected after deleting any one vertex.\footnote{By the classical Menger theorem, 2-connectivity is also equivalent to the property that the graph has at least three vertices and every two distinct vertices are joined by two paths sharing no other vertices.} 

Our main result is that every connected component of $G_p$ is 2-connected.

\begin{theorem}
\label{thm:markoff-components}
Let $p\geq5$ be a prime. Every non-loop edge of $G_p$ lies on a cycle, and moreover every connected component of $G_p$ is $2$-connected. In particular, if $G_p$ is connected, then it is $2$-connected.
\end{theorem}

Combining Theorem \ref{thm:markoff-components} with the results of Eddy, Fuchs, Litman, Martin, and Tripeny \cite{EddyEtAl2025} and Brown \cite{Brown2025} yields the $2$-connectivity of the Markoff graph $G_p$ for an explicit range of primes.

\begin{corollary}
\label{cor:large-p}
The Markoff graph $G_p$ is $2$-connected for every prime $p\in [5,10^6)\cup(3.449\cdot10^{392},+\infty)$.
\end{corollary}

We next discuss the sharpness of the vertex-connectivity in Corollary~\ref{cor:large-p}. 
For $p=5$, one can easily verify that $G_5$ is $3$-connected; see \cite[Figure 13.1]{de2024non} for an illustration of $G_5$. By contrast, $G_p$ is not $3$-connected for every prime $p\geq7$. Indeed, \cite[Lemma 2.3]{cerbu2020cycle} shows that $G_p$ has $3(p-3)$ vertices with a loop when $p\equiv3\pmod4$, and $3(p-5)$ such vertices when $p\equiv1\pmod4$. Moreover, each vertex of $G_p$ has at most one loop. Hence every vertex has either two or three distinct neighbors, and for every prime $p\ge7$ there exists a vertex with exactly two distinct neighbors. Since every $3$-connected graph has minimum degree at least $3$, it follows that $G_p$ is not $3$-connected for any prime $p\geq 7$.

\medskip

The preceding results are in fact instances of a more general phenomenon.
Our method applies uniformly to the \emph{generalized Markoff equation}
\begin{equation}
\label{eq:generalized-markoff}
x_1^2+x_2^2+x_3^2=x_1x_2x_3+k,
\end{equation}
for $k\in\FF_p$. Let
\[
X_{p,k}\coloneqq
\left\{(x_1,x_2,x_3)\in\FF_p^3:
x_1^2+x_2^2+x_3^2=x_1x_2x_3+k
\right\}
\]
It is evident that the Vieta involutions $m_1,m_2,m_3$ defined in (\ref{eqn:vieta_involution}) preserve $X_{p,k}$ for any $k\in\FF_p$. 
We therefore define the \emph{generalized Markoff graph} $G_{p,k}$ to be the undirected graph with vertex set $X_{p,k}$ whose edges are given by the three Vieta involutions. 
Note that the original Markoff graph $G_p$ is obtained from $G_{p,0}$ by removing the origin $(0,0,0)$.
For general results on the generalized Markoff equation~\eqref{eq:generalized-markoff} and the associated graph $G_{p,k}$, we refer the interested reader to \cite{cerbu2020cycle, GS2022}.

\smallskip

We now turn to the structure of the connected components of the generalized Markoff graph $G_{p,k}$.
A vertex in $G_{p,k}$ is called \emph{\textbf{axial}} if at least two of its coordinates are zero.\footnote{We remark that axial vertices exist if and only if $k=t^2$. In this case, $(\pm t,0,0)$ and their permutations are axial.}
The following theorem gives a complete description of the connected components of $G_{p,k}$, distinguishing the exceptional components containing axial vertices from all others. In particular, it yields a uniform $2$-connectivity result for every component of order at least three whenever $k\in\FF_p\setminus\{4\}$.

\begin{theorem}
\label{thm:generalized-main}
Let $p\geq5$ be a prime and let $k\in\FF_p\setminus\{4\}$. Then the following statements hold.
\begin{enumerate}
\item If a connected component $C$ of $G_{p,k}$ contains an axial vertex, then there exists $t\in\FF_p$ with $t^2=k$ such that, after permuting the coordinates if necessary,
\[
V(C)=\{(t,0,0),(-t,0,0)\}.
\]
\item If a connected component $C$ of $G_{p,k}$ contains no axial vertex, then every non-loop edge of $C$ is contained in a cycle, and moreover $C$ is $2$-connected.
\end{enumerate}
\end{theorem}

We make two remarks. First, Theorem~\ref{thm:markoff-components} follows immediately from Theorem~\ref{thm:generalized-main}.
Consider the case $k=0$. In this case, the origin $(0,0,0)$ is the unique axial vertex of $G_{p,0}$ and forms a singleton component. The original Markoff graph $G_p$ is obtained from $G_{p,0}$ by deleting the origin $(0,0,0)$, and hence every component of $G_p$ contains no axial vertices.
Theorem~\ref{thm:markoff-components} therefore follows from Theorem~\ref{thm:generalized-main}(2). 

Second, the condition $k\neq4$ is essential in Theorem~\ref{thm:generalized-main}. Indeed, if $k=4$, then $(2,1,1)$ is a solution of the generalized Markoff equation. Moreover, we have
\[
m_2(2,1,1)=m_3(2,1,1)=(2,1,1),
\qquad \mbox{and} \qquad
m_1(2,1,1)=(-1,1,1).
\]
This shows that $(2,1,1)$ has two loops and thus a unique distinct neighbor. 
Therefore, the connected component (say $C$) containing $(2,1,1)$ cannot be $2$-connected. 
On the other hand, this component $C$ clearly has more than two vertices, because the vertex
\(m_2(-1,1,1)=(-1,-2,1)\)
is distinct from both $(2,1,1)$ and $(-1,1,1)$, for any $p\geq5$. Putting everything together, we see Theorem~\ref{thm:generalized-main} cannot hold for the case $k=4$.

\medskip

The proof of Theorem~\ref{thm:generalized-main} brings together ideas from number theory and graph theory. On the number-theoretic side, it exploits the arithmetic of the generalized Markoff equation over $\mathbb{F}_p$ and its Vieta involutions; on the graph-theoretic side, it uses graph flows to control bridges in certain quotient graphs defined over $G_{p,k}$.\footnote{Recall that an edge of a graph is called a \emph{bridge} if deleting it increases the number of connected components.}
We refer the reader to Subsection~\ref{Subsec:overview} for an overview of the proof of the latter part of the argument, which constitutes the main technical contribution of this paper.

\medskip

The remainder of the paper is organized as follows. In Section \ref{sec:quotient}, we investigate the quotient graph $\overline{G}_{p,k}$, which captures the double sign-change symmetries of the generalized Markoff graph $G_{p,k}$. We also identify a distinguished class of quotient vertices, called \emph{cusps}, whose local connectivity and lifting properties play a key role in our proof.
In Section \ref{sec:bridges}, we introduce the notion of \emph{graph flows} and use it, together with other graph-theoretic and number-theoretic arguments, to analyze the \emph{bridges} of $\overline{G}_{p,k}$.
We then establish the central proposition concerning the lifting of these bridges, namely Proposition~\ref{prop:generalized-bridge}.
Finally, in Section \ref{sec:main_thm}, we complete the proof of Theorem~\ref{thm:generalized-main}.

\section{The quotient graph}
\label{sec:quotient}

Define the \emph{double sign-change automorphisms} $\sigma_1,\sigma_2,\sigma_3\colon X_{p,k}\to X_{p,k}$ by
\[
\begin{aligned}
\sigma_1(x_1,x_2,x_3)&\coloneqq(x_1,-x_2,-x_3),\\
\sigma_2(x_1,x_2,x_3)&\coloneqq(-x_1,x_2,-x_3),\\
\sigma_3(x_1,x_2,x_3)&\coloneqq(-x_1,-x_2,x_3).
\end{aligned}
\]
Set
\[
K\coloneqq\{1,\sigma_1,\sigma_2,\sigma_3\}.
\]
It is easy to verify the relations $\sigma_i^2=1$ for $i\in\{1,2,3\}$ and $\sigma_i\sigma_j=\sigma_j\sigma_i=\sigma_k$ for distinct $i,j,k\in\{1,2,3\}$. Thus we have a group isomorphism
$$
K\cong(\ZZ/2\ZZ)\times(\ZZ/2\ZZ).
$$
Moreover, every element of $K$ commutes with each of the Vieta involutions $m_1,m_2,m_3$.

Let $X_{p,k}^{\circ}$ be the set of non-axial vertices in $G_{p,k}$. The group $K$ acts freely on $X_{p,k}^{\circ}$. Let $G_{p,k}^{\circ}$ denote the subgraph of $G_{p,k}$ induced by $X_{p,k}^{\circ}$. Set
\[
Y_{p,k}\coloneqq K\backslash X_{p,k}^{\circ}=\{Kx: x\in X_{p,k}^{\circ}\},
\]
be the set of $K$-orbits in $X_{p,k}^{\circ}$, and let
\[
\pi_k\colon X_{p,k}^{\circ}\longrightarrow Y_{p,k}
\]
be the canonical projection. We now define the \emph{quotient graph} $\overline{G}_{p,k}$ as follows. 
Its vertex set is the set $Y_{p,k}$ of $K$-orbits, with two $K$-orbits $Kx,Ky\in Y_{p,k}$ adjacent if $m_i(x)=\sigma(y)$ for some $i\in\{1,2,3\}$ and $\sigma\in K$.
Since the group $K$ commutes with Vieta involutions, this is independent of the choice of representative. 
An edge (resp. vertex) of $\overline{G}_{p,k}$ is called a \emph{quotient edge} (resp. \emph{quotient vertex}).
In fact, if $Kx\sim Ky$ is a non-loop quotient edge in $\overline{G}_{p,k}$, then there is an induced matching of size $4$ between $Kx$ and $Ky$ in $G_{p,k}^\circ$, where each of $Kx$ and $Ky$ is viewed as a four-vertex subset of $X_{p,k}^{\circ}$. Moreover, all edges in this induced matching have the same ``label'', i.e., they are defined by the same Vieta involution $m_i$ (see Lemma~\ref{lem:generalized-labels}).

We now introduce a special kind of vertex in $G_{p,k}^{\circ}$ and $\overline{G}_{p,k}$ that plays an important role in our proofs. A vertex of $G_{p,k}^{\circ}$ is a \textbf{\emph{cusp}} if exactly one coordinate is zero. Since the axial vertices have already been removed, every remaining vertex in $G_{p,k}^{\circ}$, which we call \emph{regular}, has all three coordinates nonzero.
For a quotient vertex $Kx$ of $\overline{G}_{p,k}$, we call $Kx$ a \emph{cusp} (resp. a \emph{regular vertex}) if one, equivalently every, vertex in $Kx$ is a cusp (resp. a regular vertex) in $G_{p,k}^{\circ}$. This definition is independent of the choice of representative, because each automorphism of $K$ preserves the vanishing or non-vanishing properties of each coordinate.

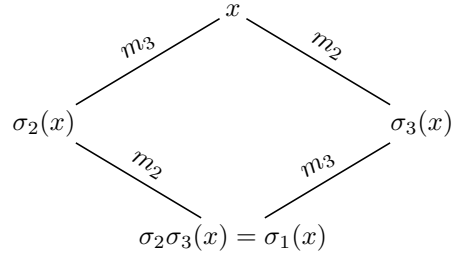
\begin{figure}[H]
\centering
\begin{tikzpicture}[
vertex/.style={inner sep=2pt},
edge/.style={semithick},
edge label/.style={midway, sloped, above}
]
\node[vertex] (x)   at (0,1.5)  {$x$};
\node[vertex] (s3)  at (2.5,0)    {$\sigma_3(x)$};
\node[vertex] (s23) at (0,-1.5) {$\sigma_2\sigma_3(x)=\sigma_1(x)$};
\node[vertex] (s2)  at (-2.5,0)   {$\sigma_2(x)$};

\draw[edge] (x)   -- node[edge label] {$m_2$} (s3);
\draw[edge] (s3)  -- node[edge label] {$m_3$} (s23);
\draw[edge] (s23) -- node[edge label] {$m_2$} (s2);
\draw[edge] (s2)  -- node[edge label] {$m_3$} (x);
\end{tikzpicture}
\caption{The $4$-cycle in $G_{p,k}^{\circ}$ formed by the four vertices over a cusp $Kx$ of $\overline{G}_{p,k}$ with $x=(0,x_2,x_3)$.}
\label{fig:four-cycle}
\end{figure}

\begin{lemma}
\label{lem:cusp4cycle}
Let $Kx$ be a cusp in $\overline{G}_{p,k}$. Then the four vertices in $Kx$ form a $4$-cycle in $G_{p,k}^\circ$.
\end{lemma}
\begin{proof}
Without loss of generality, we assume that $x=(0,x_2,x_3)$ with $x_2,x_3\neq0$. It is easy to verify that
$$
m_2(x)=(0,-x_2,x_3)=\sigma_3(x)
$$
and
$$
m_3(x)=(0,x_2,-x_3)=\sigma_2(x).
$$
Thus, we have
$$
\sigma_2\sigma_3(x)=\sigma_2m_2(x)=m_2\sigma_2(x),
$$
and
$$
\sigma_2\sigma_3(x)=\sigma_3\sigma_2(x)=\sigma_3m_3(x)=m_3\sigma_3(x).
$$
Since $\sigma_1=\sigma_2\sigma_3=\sigma_3\sigma_2$, we obtain the $4$-cycle in Figure \ref{fig:four-cycle} consisting of the four vertices in $Kx$.
\end{proof}

Next we prove a lifting property of connected subgraphs containing cusps.

\begin{lemma}
\label{lem:cuspconnected}
Let $\pi:=\pi_k$.
If a connected subgraph $H\subseteq\overline{G}_{p,k}$ contains a cusp $Kc$, then the subgraph of $G_{p,k}^\circ$ induced by $\pi^{-1}(V(H))$, denoted by $G_{p,k}^\circ[\pi^{-1}(V(H))]$, is connected.
\end{lemma}
\begin{proof}
Take any vertex $x$ in the induced subgraph $G_{p,k}^\circ[\pi^{-1}(V(H))]$. By definition, we have $Kx\in V(H)$. Since $H$ is connected, we choose a path in $H$ from $Kx$ to the cusp $Kc$, and lift it to a path in $G_{p,k}^\circ$ from $x$ to $\sigma(c)$ for some $\sigma\in K$. Thus, every vertex in the induced subgraph $G_{p,k}^\circ[\pi^{-1}(V(H))]$ is connected to one vertex of $Kc$ in $G_{p,k}^\circ$. Now, by applying Lemma \ref{lem:cusp4cycle}, the four vertices of $Kc$ form a 4-cycle in $G_{p,k}^\circ$, in particular, a connected subgraph of $G_{p,k}^\circ$. It follows that any two vertices of $G_{p,k}^\circ[\pi^{-1}(V(H))]$ are connected to each other via the $4$-cycle. Hence, we see that the induced subgraph $G_{p,k}^\circ[\pi^{-1}(V(H))]$ is connected.
\end{proof}

For convenience, we label each edge of the generalized Markoff graph $G_{p,k}$ by $i\in \{1,2,3\}$ if it is defined by the Vieta involution $m_i$. In the next lemma, we show that these labels are inherited unambiguously by every quotient edge that is either a non-loop or incident to a regular vertex of the quotient graph $\overline{G}_{p,k}$, except when $k=4$.

\begin{lemma}
\label{lem:generalized-labels}
Assume that $k\in\FF_p\setminus\{4\}$ and consider a vertex $x$ in $G_{p,k}^{\circ}$. Then we have the following.
\begin{itemize}
\item[(1)] If $x$ is regular, then $Km_i(x)\neq Km_j(x)$ for distinct $i,j\in\{1,2,3\}$.
\item[(2)] If $x$ is regular and $Km_i(x)=Kx$ for some $i\in\{1,2,3\}$, then $m_i(x)=x$.
\item[(3)] If $x=(x_1,x_2,x_3)$ is a cusp with $x_i=0$, then $Kx$ has a unique non-loop quotient edge, which is the projection of the label-$i$ edge incident to $x$.
\end{itemize}
Consequently, every quotient edge of $\overline{G}_{p,k}$ that is either a non-loop or incident to a regular vertex of $\overline{G}_{p,k}$ has a unique label.
\end{lemma}

\begin{proof}
Write $x=(x_1,x_2,x_3)$. 
For part (1), without loss of generality, we may assume for a contradiction that for some regular $x$, 
\[
Km_1(x)=Km_2(x).
\]
Then there exists some $\kappa\in K$ such that
\[
m_1(x)=\kappa m_2(x).
\]
If $\kappa\in \{\sigma_1, \sigma_2\}$, comparison of the third coordinates forces $x_3=0$, a contradiction; if $\kappa=\sigma_3$, comparison of the first two coordinates gives $x_2x_3=x_1x_3=0$, again a contradiction. Hence we have $\kappa=1$. Comparing the first two coordinates gives
\[
x_2x_3=2x_1
\qquad\text{and}\qquad
x_1x_3=2x_2.
\]
Since $x$ is regular, it follows that $x_3^2=4$. Write $x_3=2\varepsilon$ with $\varepsilon\in\{\pm1\}$, and obtain $x_1=\varepsilon x_2$. Substitution into the generalized Markoff equation gives
\[
k=x_1^2+x_2^2+x_3^2-x_1x_2x_3=(\varepsilon x_2)^2+x_2^2+(2\varepsilon)^2-(\varepsilon x_2)\cdot x_2\cdot (2\varepsilon)=4.
\]
This contradicts to our assumption. This proves part (1).

For part (2), without loss of generality, we can assume that $i=1$, i.e., $m_1(x)=\kappa x$ for some $\kappa\in K$. Direct comparison shows that every non-identity choice of $\kappa$ forces $x_2=0$ or $x_3=0$. Since $x$ is regular, we must have $\kappa=1$, and hence $m_1(x)=x$. This proves part (2).

For part (3), without loss of generality, we assume that $x=(0,x_2,x_3)$ is a cusp with $x_2x_3\neq0$. Then we have $m_2(x)=\sigma_3(x)$ and $m_3(x)=\sigma_2(x)$, so labels $2$ and $3$ give rise to two quotient loops. Since $x_2x_3\neq0$, we see $m_1(x)$ is a regular vertex. 
Thus the edge connecting $x$ and $m_1(x)$ is a non-loop with label $1$, which gives rise to a non-loop quotient edge. This proves part (3).

The last assertion follows from parts (1) and (3) immediately.
\end{proof}

We end this section with the following lemma, concerning all connected components of the generalized Markoff graph $G_{p,k}$ that contain an axial vertex. This proves Theorem \ref{thm:generalized-main}(1).

\begin{lemma}
\label{lem:small-components}
Assume that $k\in\FF_p\setminus\{4\}$. Every connected component of $G_{p,k}$ containing an axial vertex, after permuting the coordinates if necessary, is of the form
\[
\{(t,0,0),(-t,0,0)\},
\]
with $t\in\FF_p$ such that $t^2=k$. Every other connected component of $G_{p,k}$ is contained in $G_{p,k}^{\circ}$ and has at least $3$ vertices. 
\end{lemma}

\begin{proof}
After permuting coordinates, we assume that the axial vertex is $(t,0,0)$. The generalized Markoff equation gives $t^2=k$, and
\[
m_1(t,0,0)=(-t,0,0),
\qquad
m_2(t,0,0)=m_3(t,0,0)=(t,0,0).
\]
Thus, the connected component of $G_{p,k}$ containing this axial vertex is of the form
\[
\{(t,0,0),(-t,0,0)\}.
\]

Let $C$ be a connected component without an axial vertex, and choose any $x\in V(C)$. Suppose first that $x$ is a cusp. Without loss of generality, assume that $x=(0,x_2,x_3)$ with $x_2x_3\neq0$. Then, since $p$ is odd,
it is evident that $x, m_2(x)=(0,-x_2,x_3)$ and $m_3(x)=(0,x_2,-x_3)$ are three distinct vertices of $C$.

Suppose now that $x$ is regular. By Lemma \ref{lem:generalized-labels}(1), the three $K$-orbits $Km_1(x),\ Km_2(x)$ and $Km_3(x)$ are pairwise distinct. In particular, this implies that $m_1(x),m_2(x)$ and $m_3(x)$ are three distinct vertices in $C$. In either case, we have $|V(C)|\geq3$.
\end{proof}

\section{Lifting quotient bridges}
\label{sec:bridges}

The aim of this section is to prove the following key lifting proposition for quotient bridges. For an edge $e=xy$ of $G_{p,k}^{\circ}$, let $\overline{e}$ denote the corresponding quotient edge in $\overline{G}_{p,k}$ joining quotient vertices $Kx$ and $Ky$.

\begin{proposition}
\label{prop:generalized-bridge}
Assume that $k\in\FF_p\setminus\{4\}$. Let $e$ be a non-loop edge of $G_{p,k}^{\circ}$. If the quotient edge $\overline{e}$ is a bridge of $\overline{G}_{p,k}$, then $e$ is contained in a cycle of $G_{p,k}^{\circ}$.
\end{proposition}

\begin{figure}[t]
  \centering
  \begin{tikzpicture}[
      x=.75cm,
      y=.75cm,
      font=\small,
      line cap=round,
      line join=round,
      a region/.style={draw=cyan!55, fill=cyan!55, fill opacity=.22, line width=.70pt},
      b region/.style={draw=orange!65, fill=orange!65, fill opacity=.22, line width=.70pt},
      preimage frame/.style={draw=black, dashed, fill=none, line width=.70pt},
      edge/.style={draw=black, line width=.85pt},
      edge vertex/.style={circle, draw=black, fill=white, line width=.75pt, inner sep=1.35pt},
      connecting path/.style={draw=black, line width=1.10pt
    },
      hollow arrow/.style={draw=black, fill=none, line width=.70pt, line cap=butt, line join=miter},
      graph arrow/.style={ ->, draw=black, line width=.75pt}
    ]

    \def\componentoutline{%
      (-1.68,.52)
      .. controls (-1.76,.25) and (-1.74,-.28) .. (-1.50,-.51)
      .. controls (-1.15,-.73) and (-.72,-.63) .. (-.39,-.68)
      .. controls (-.02,-.74) and (.20,-.56) .. (.58,-.63)
      .. controls (.96,-.70) and (1.35,-.55) .. (1.52,-.29)
      .. controls (1.68,-.08) and (1.72,.25) .. (1.54,.47)
      .. controls (1.35,.69) and (.92,.60) .. (.88,.68)
      .. controls (.22,.77) and (-.05,.60) .. (-.42,.68)
      .. controls (-.82,.76) and (-1.20,.64) .. (-1.68,.52)
      -- cycle
    }

    \begin{scope}[shift={(-6.25,0)}]
      \path[a region]
        (-2.15,.10) ellipse [x radius=1.45,y radius=.90];
      \path[b region]
        ( 2.15,.10) ellipse [x radius=1.45,y radius=.90];

      \coordinate (ubar) at (-1.35,.10);
      \coordinate (vbar) at ( 1.35,.10);
      \draw[edge]
        (ubar) -- node[above=1pt,fill=white,inner sep=1pt]
        {$\overline{e}$} (vbar);
      \node[edge vertex] at (ubar) {};
      \node[edge vertex] at (vbar) {};

      \node[left=1pt]  at (ubar) {$Ku$};
      \node[right=1pt] at (vbar) {$Kv$};
      \node at (-2.15,-1.10) {$A$};
      \node at ( 2.15,-1.10) {$B$};
    \end{scope}

    \begin{scope}[shift={(6.00,0)}]
      \path[preimage frame] (-5.00,-2.10) rectangle (-.55,2.10);
      \path[preimage frame] ( .55,-2.10) rectangle ( 5.00,2.10);

      \begin{scope}[shift={(-2.75,.88)},xscale=1.15,yscale=1.24]
        \path[a region] \componentoutline;
      \end{scope}
      \begin{scope}[shift={(-2.75,-.98)},xscale=1.15,yscale=1.24]
        \path[a region] \componentoutline;
      \end{scope}
      \begin{scope}[shift={(2.75,.88)},xscale=1.15,yscale=1.24]
        \path[b region] \componentoutline;
      \end{scope}
      \begin{scope}[shift={(2.75,-.98)},xscale=1.15,yscale=1.24]
        \path[b region] \componentoutline;
      \end{scope}

      \coordinate (u)      at (-1.35, 1.30);
      \coordinate (v)      at ( 1.35, 1.30);
      \coordinate (su)     at (-1.35,  .40);
      \coordinate (sv)     at ( 1.35,  .40);
      \coordinate (aThree) at (-1.35, -.50);
      \coordinate (bThree) at ( 1.35, -.50);
      \coordinate (aFour)  at (-1.35,-1.40);
      \coordinate (bFour)  at ( 1.35,-1.40);

      \draw[edge]
        (u) -- node[above=1pt,fill=white,inner sep=1pt] {$e$} (v);
      \draw[edge]
        (su) -- node[above=1pt,fill=white,inner sep=1pt]
        {$\sigma(e)$} (sv);
      \draw[edge] (aThree) -- (bThree);
      \draw[edge] (aFour)  -- (bFour);

      \draw[connecting path]
        (u)
        .. controls (-1.70,1.32) and (-1.75,1.53) .. (-2.13,1.49)
        .. controls (-2.50,1.45) and (-2.43,1.17) .. (-2.80,1.15)
        .. controls (-3.18,1.13) and (-3.14,1.40) .. (-3.51,1.35)
        .. controls (-3.92,1.30) and (-3.99,.96) .. (-3.65,.84)
        .. controls (-3.34,.73) and (-3.46,.48) .. (-3.12,.45)
        .. controls (-2.78,.42) and (-2.74,.69) .. (-2.40,.65)
        .. controls (-2.05,.61) and (-2.04,.35) .. (-1.70,.31)
        .. controls (-1.52,.29) and (-1.47,.36) .. (su);
      \draw[connecting path]
        (v)
        .. controls (1.64,1.31) and (1.78,1.49) .. (2.10,1.42)
        .. controls (2.42,1.35) and (2.33,1.08) .. (2.70,1.05)
        .. controls (3.03,1.02) and (3.15,1.27) .. (3.53,1.19)
        .. controls (3.90,1.11) and (4.05,.79) .. (3.72,.66)
        .. controls (3.41,.54) and (3.52,.31) .. (3.15,.30)
        .. controls (2.78,.29) and (2.70,.56) .. (2.35,.52)
        .. controls (1.96,.48) and (1.82,.32) .. (sv);

      \foreach \pt in {u,v,su,sv,aThree,bThree,aFour,bFour}
        \node[edge vertex] at (\pt) {};

      \node[above=0pt] at (u)  {$u$};
      \node[above=0pt]  at (v)  {$v$};
      \node[above=0pt] at (su) {$\sigma(u)$};
      \node[above=0pt]  at (sv) {$\sigma(v)$};
      \node at (-4.10,.55) {$P_A$};
      \node at ( 4.20,1.35) {$P_B$};

      \node[below=2pt,align=center] at (-2.75,-2.10)
        {$G_{p,k}^{\circ}[\pi_k^{-1}(V(A))]$};
      \node[below=2pt,align=center] at ( 2.75,-2.10)
        {$G_{p,k}^{\circ}[\pi_k^{-1}(V(B))]$};
    \end{scope}

    \node at (-.95,.60) {lifting};
    \draw[hollow arrow]
      (-2.05,-.05) -- (-.20,-.05) -- (-.20,-.15) --
      ( .35, 0)    -- (-.20, .15) -- (-.20, .05) --
      (-2.05, .05) -- cycle;

    \node[below=2pt,align=center,font=\normalsize] at (-6.00,-2.10)
      {$\overline{G}_{p,k}$};
   
  \end{tikzpicture}

  \caption{Suppose that, for some $\sigma\in K\setminus\{1\}$, there exists a path $P_A$ joining $u$ to $\sigma(u)$ in $G_{p,k}^{\circ}[\pi_k^{-1}(V(A))]$ and a path $P_B$ joining $v$ to $\sigma(v)$ in $G_{p,k}^{\circ}[\pi_k^{-1}(V(B))]$. Then $e\cup P_B\cup\sigma(e)\cup P_A$ forms a cycle containing $e$.}
  \label{fig:bridgelift}
\end{figure}
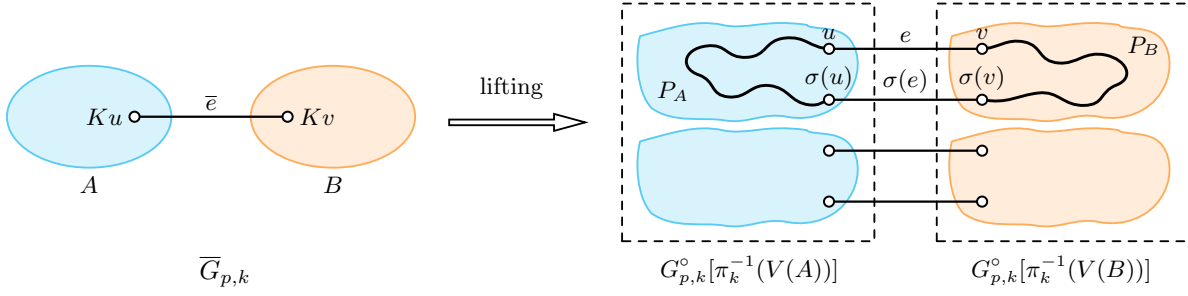

We clarify some notation.
Throughout, the notation $Kx$ will be used in two closely related but importantly distinct senses: as a quotient vertex of $\overline{G}_{p,k}$ and as the corresponding four-vertex subset of $G_{p,k}^{\circ}$.
Let $Q$ be a connected component of a graph, and let $f$ be a bridge of $Q$. Then there exist unique connected subgraphs $X$ and $Y$ of $Q$ such that $Q-f=X\sqcup Y$, which we refer to as the two \emph{shores} of the bridge $f$. 

\subsection{Proof Overview}\label{Subsec:overview}
The proof of Proposition~\ref{prop:generalized-bridge} is based on two complementary ideas. 
The first idea is a combinatorial construction of a cycle containing $e$, formulated in algebraic terms; see Figure~\ref{fig:bridgelift} for an illustration. Let $e=uv$ be a non-loop edge of $G_{p,k}^{\circ}$, and suppose its quotient $\overline{e}$ is a bridge of $\overline{G}_{p,k}$ with shores $A$ and $B$. To construct a cycle containing $e$, we seek for a non-identity element $\sigma\in K$ such that the endpoints of $e$ and $\sigma(e)$ can be connected in the lifts of the two shores of $\overline{e}$, respectively. This motivates us to record,  for a given vertex $x\in X_{p,k}^{\circ}$, which elements in $Kx$ can be reached from $x$ in the lift of $\overline{G}_{p,k}$. To be precise, let $R$ be a connected subgraph of $\overline{G}_{p,k}$. For $x\in X_{p,k}^{\circ}$ satisfying $Kx\in V(R)$, define
\begin{equation}\label{equ:H_R}
H_R(x)\coloneqq
\left\{\sigma\in K:
x\text{ and }\sigma(x)\text{ are joined by a path in }
G_{p,k}^{\circ}[\pi_k^{-1}(V(R))]
\right\}.
\end{equation}
With this notation, our goal is to show that $H_A(u)\cap H_B(v)$ contains a non-identity element $\sigma$.
The main difficulty is to rule out the possibility that $H_R(x)$ consists solely of the identity element for each $(R,x)\in \{(A,u),(B,v)\}$.

Our second idea is to use the fundamental graph-theoretic concept of \emph{flow} to analyze the quotient bridges. 
Let $G=(V,E)$ be a directed graph and let $\Gamma$ be an abelian group. An assignment $\phi:E\rightarrow\Gamma$ is called a $\Gamma$-\emph{flow} if the \emph{divergence} defined by
$$
\operatorname{div}(x)\coloneqq\sum_{f\in\delta^+(x)}\phi(f)-\sum_{f\in\delta^-(x)}\phi(f),
$$
equals zero at every vertex $x\in V$, where we denote by $\delta^+(x)$ (resp. $\delta^-(x)$) the set of edges directed out of (resp. directed into) $x$. A $\Gamma$-flow $\phi$ is \emph{nowhere-zero} if $\phi(f)\neq 0$ for any $f\in E$. The use of flows to study the quotient bridges is inspired by the classical result in graph theory stating that if a directed graph admits a nowhere-zero $\Gamma$-flow for some abelian group $\Gamma$, then it contains no bridges. To see this, suppose to the contrary that $\phi$ is a nowhere-zero $\Gamma$-flow on a directed graph $G$ containing a bridge $f$. Let $X$ and $Y$ be the two shores of $f$. Then summing $\operatorname{div}(x)$ over all vertices $x\in V(X)$ and noting that the contributions from two endpoints of every oriented edge other than $f$ cancel with each other, we obtain $\phi(f)=0$, contradicting the nowhere-zero property of $\phi$. 
We use flows twice below, in Lemmas~\ref{lem:diagonal-bridge} and \ref{lem:nontrivial-signs}.

\smallskip

We now briefly explain how these two seemingly disparate ideas work together. 
Consider a non-loop edge $e$ of $G_{p,k}^\circ$ whose quotient edge $\overline{e}$ is a bridge of the quotient graph $\overline{G}_{p,k}$.
We construct flows on an orientation of $\overline{G}_{p,k}$, using the algebraic information associated with the vertices and edges of $G^\circ_{p,k}$.\footnote{A cautious reader may notice that here we refer to vertices and edges of $G^\circ_{p,k}$, rather than those of $\overline{G}_{p,k}$. This distinction is essential: a quotient edge need not have a legal label, and Lemma~\ref{lem:generalized-labels} is precisely what guarantees the existence of the legal labels needed in our construction.}
Although these flows need not be nowhere-zero, the preceding argument still forces the flow value on the quotient bridge $\overline{e}$ to vanish, which in turn yields further algebraic information.
More specifically, let $R$ be a shore of $\overline{e}$ and suppose that $R$ is cusp-free. One such flow typically yields algebraic relations at the endpoint $x$ of $e$ lying in the lift of $R$, thereby restricting the possibilities for $H_R(x)$. 
In particular, our second flow (i.e., Lemma~\ref{lem:nontrivial-signs}) rules out the possibility that $H_R(x)$ consists solely of the identity element. 
Together with the properties of cusps established in Lemma~\ref{lem:cuspconnected}, these facts allow us to apply the first idea and construct a desired cycle containing $e$. 
See Figure~\ref{fig:proofsketch} for a roadmap of the proof of Proposition~\ref{prop:generalized-bridge}

\begin{figure}[t]
\centering
\begin{tikzpicture}[x=1cm,y=1cm,
  pfimplies/.style={line width=.35pt,double equal sign distance=3mm,-{Latex[line width=.35pt,length=3mm,fill=white]}},
  pfcase/.style={line width=.45pt},
  hrarrow/.style={line width=.45pt,-{Latex}},
  pfbase/.style={draw,rounded corners=1.2pt,line width=.45pt,align=center,font=\small,inner xsep=5pt, inner ysep=3.2pt},
  pfroot/.style={pfbase,text width=11.6cm,line width=.7pt},
  pfbranch/.style={pfbase,text width=4.5cm},
  pfresult/.style={pfbase,text width=5.35cm,line width=.7pt},
  pffinal/.style={pfbase,line width=.8pt},
  pflemmalabel/.style={font=\small,fill=white,inner xsep=2pt,inner ysep=.8pt}
]

\node[pfbranch, text width=4.5cm] (fixshore) at (-0.2,-1.30-0.1)
  {Fix $(R,x)\in\{(A,u),(B,v)\}$.};

\node[pfbranch,text width=4.5cm] (cusp) at (-4.30+0.1,-3.05+0.8) {$R$ contains a cusp.};
\node[pfbranch,text width=4.5cm] (nocusp) at (3.30+0.5,-3.05+0.8) {$R$ is cusp-free.};

\node[pfbranch,text width=4.5cm] (connected) at (-4.30+0.1,-4.55+0.8)
  {$G_{p,k}^{\circ}[\pi_k^{-1}(V(R))]$ is connected.};
\node[pfbranch,text width=4.5cm] (coordinates) at (3.30+0.5,-4.55+0.8)
  {$x_2^2=x_3^2\ne0$.};

\node[pfbranch,text width=4.5cm] (leftsigma) at (-4.30+0.1,-6.05+1.2) {$\sigma_1\in H_R(x)$.};
\node[pfbranch,text width=4.5cm] (possibilities) at (3.30+0.5,-6.05+1.2)
  {$H_R(x)=\{1\},\ \{1,\sigma_1\}\text{ or }K$.};

\node[pfbranch,text width=4.5cm] (singleton) at (3.30+0.5,-7.55+1.7)
  {$H_R(x)=\{1\}$.};

\node[pfbranch,text width=4.5cm] (bothshores) at (-4.30+0.1,-9.15+2.5)
  {Applying to both shores gives $\sigma_1\in H_A(u)\cap H_B(v)$.};
\node[pfbranch,text width=2.8cm] (exceptional) at (3.30-2-1,-9.15+2.4)
  {$x_2$ or $x_3=\pm 2$.};
\node[pfbranch,text width=3.5cm] (generic) at (3.30+0.5+0.5,-9.15+2.4)
  {$x_2^2=x_3^2\ne4$.};
\node[pfbranch,text width=3.5cm] (onesheet) at (3.30+0.5+0.5,-10.55+2.7)
  {$R$ has a one-sheet lift.};

\node[pffinal,text width=4.5cm] (maincycle) at (-4.30+0.1,-12.05+2.8)
  {$e$ lies on a cycle.};
\node[pffinal,text width=2.8cm,font=\footnotesize] (cycleleaf) at (3.30-2-1,-12.05+2.8)
  {$e$ lies on a cycle.};
\node[pffinal,text width=3.5cm,font=\footnotesize] (contradiction) at (3.30+0.5+0.5,-12.05+2.8)
  {A contradiction.};

\coordinate (cuspsplit) at ($(fixshore.south)+(0,-.12)$);
\draw[pfcase] (fixshore.south) -- (cuspsplit);
\draw[pfcase] (cuspsplit) -- (cusp.north |- cuspsplit) -- (cusp.north);
\draw[pfcase] (cuspsplit) -- (nocusp.north |- cuspsplit) -- (nocusp.north);

\draw[hrarrow] (possibilities.west) -- (leftsigma.east);

\draw[hrarrow] (possibilities.south) -- (singleton.north);

\coordinate (coordsplit) at ($(singleton.south)+(0,-.12)$);
\draw[pfcase] (singleton.south) -- (coordsplit);
\draw[pfcase] (coordsplit) -- (exceptional.north |- coordsplit) -- (exceptional.north);
\draw[pfcase] (coordsplit) -- (generic.north |- coordsplit) -- (generic.north);

\draw[pfimplies] (cusp) -- node[pflemmalabel,right=3pt] {Lemma~\ref{lem:cuspconnected}} (connected);
\draw[pfimplies] (connected) -- node[pflemmalabel,right=3pt] {By definition of $H_R(x)$} (leftsigma);
\draw[pfimplies] (nocusp) -- node[pflemmalabel,right=3pt,align=center] {Lemma~\ref{lem:diagonal-bridge}\\[-1pt](graph flow)} (coordinates);
\draw[pfimplies] (coordinates) -- node[pflemmalabel,right=3pt] {$\ $Lemma~\ref{lem:diagonal-symmetry}} (possibilities);
\draw[pfimplies] (leftsigma.south) -- (leftsigma.south |- bothshores.north);

\draw[pfimplies] (bothshores) --
  node[pflemmalabel,left=3pt,align=center] {Constructing\\[-1pt] a cycle as\\[-1pt] in Figure~\ref{fig:bridgelift}}
  (maincycle);
\draw[pfimplies] (exceptional) -- node[pflemmalabel,left=3pt] {Lemma~\ref{lem:pole-cycle}} (cycleleaf);
\draw[pfimplies] (generic) -- node[pflemmalabel,right=3pt] {$\ $Lemma~\ref{lem:reachable-signs}} (onesheet);
\draw[pfimplies] (onesheet) -- node[pflemmalabel,right=3pt,align=center] {Lemma~\ref{lem:nontrivial-signs}\\[-1pt](graph flow)} (contradiction);
\end{tikzpicture}
\caption{A roadmap of the proof of Proposition~\ref{prop:generalized-bridge}.}
\vspace{-2mm}
\captionsetup{width=0.75\textwidth}
\caption*{Here, we work in the following setting. Let $e=uv$ be a label-$1$ non-loop edge of $G_{p,k}^\circ$ such that the quotient edge $\overline{e}$ joining $Ku$ and $Kv$ is a bridge of $\overline{G}_{p,k}$. Let $A$ and $B$ be the two shores of $\overline{e}$.}
\label{fig:proofsketch}
\end{figure}

\subsection{The Proof}
Before presenting the proof of Proposition~\ref{prop:generalized-bridge}, we establish a sequence of lemmas.

We first introduce several functions that provide the essential ingredients for constructing graph flows. 
For a vertex $x=(x_1,x_2,x_3)\in X_{p,k}^{\circ}$, we define
\[
\Delta_1(x)\coloneqq2x_1-x_2x_3,\quad \Delta_2(x)\coloneqq2x_2-x_1x_3,\quad \Delta_3(x)\coloneqq2x_3-x_1x_2.
\]
Thus $m_i(x)=x$ if and only if $\Delta_i(x)=0$. For $i\in\{1,2,3\}$, we define the following functions
\[
S_1(x)\coloneqq
\frac{\Delta_1(x)\cdot (x_2^2-x_3^2)}{x_2x_3},\quad
S_2(x)\coloneqq
\frac{\Delta_2(x)\cdot (x_3^2-x_1^2)}{x_1x_3},\quad
S_3(x)\coloneqq
\frac{\Delta_3(x)\cdot (x_1^2-x_2^2)}{x_1x_2},
\]
whenever the corresponding denominator is nonzero. 

In the first lemma, we collect several useful properties of the functions $S_i$, which may be viewed, in some sense, as necessary conditions for constructing the desired flows.

\begin{lemma}
\label{lem:first-flow}
The following statements hold.
\begin{itemize}
\item[(1)] Wherever $S_i(x)$ is defined, $S_i(m_i(x))$ is also defined and we have
\[
S_i(m_i(x))=-S_i(x).
\]
\item[(2)] Wherever $S_i(x)$ is defined, $S_i(\sigma(x))$ is also defined and for every $\sigma\in K$, we have
\[
S_i(\sigma(x))=S_i(x).
\]
\item[(3)] If $x$ is regular, then
\[
S_1(x)+S_2(x)+S_3(x)=0.
\]
\end{itemize}
\end{lemma}

\begin{proof}
(1) By direct calculations, we have $\Delta_i(x)=-\Delta_i(m_i(x))$, which, together with the fact that the involution $m_i$ fixes the other two coordinates, proves (1). 

(2) Every non-identity automorphism $\sigma\in K$ changes the signs of two coordinates of $x$. It is straightforward to check that the sign-change of the denominator of $S_i$ cancels with that of its numerator. This proves (2).

(3) Multiplying the left-hand side of the desired identity by $x_1x_2x_3$, we obtain
\[
\bigl(S_1(x)+S_2(x)+S_3(x)\bigr)\cdot x_1x_2x_3=x_1\Delta_1(x)(x_2^2-x_3^2)
+x_2\Delta_2(x)(x_3^2-x_1^2)
+x_3\Delta_3(x)(x_1^2-x_2^2).
\]
Substituting $x_i\Delta_i(x)=2x_i^2-x_1x_2x_3$ for $i\in\{1,2,3\}$, this equals
\[
2\bigl(x_1^2(x_2^2-x_3^2)
+x_2^2(x_3^2-x_1^2)
+x_3^2(x_1^2-x_2^2)\bigr)
-x_1x_2x_3\bigl((x_2^2-x_3^2)
+(x_3^2-x_1^2)
+(x_1^2-x_2^2)\bigr)=0,
\]
as desired.
\end{proof}

We now use these $K$-invariant functions $S_i(x)$ to construct the first flow in the proof, thereby establishing certain algebraic relations among the corresponding vertices. 

\begin{lemma}
\label{lem:diagonal-bridge}
Assume that $k\in\FF_p\setminus\{4\}$. Let $\overline{e}$ be a bridge with label $1$ in a connected component $Q$ of $\overline{G}_{p,k}$, and suppose that $A$ and $B$ are the two \emph{shores} of $\overline{e}$, i.e.,
\[
Q-\overline{e}=A\sqcup B.
\]
If a shore $R\in\{A,B\}$ contains no cusp, then the endpoint $Kx\in V(R)$ of $\overline{e}$ satisfies
\(x_2^2=x_3^2.\)
\end{lemma}
\begin{proof}
Since $R$ contains no cusp, every vertex of $R$ is regular. 
Choose an orientation for each quotient edge $\overline{f}\in E(R)\cup\{\overline{e}\}$ such that $\overline{e}$ is directed out of $R$. Then every oriented $\overline{f}$ has a regular initial endpoint, say $Ky\in V(R)$. 
By Lemma~\ref{lem:generalized-labels}, $\overline{f}$ has a unique label $i$, and we set
\[
\phi(\overline{f})\coloneqq S_i(y).
\]
Lemma \ref{lem:first-flow}(2) shows that this is independent of the representative $y$. And the divergence 
$$
\operatorname{div}(Ky)=\sum_{\overline{f}\in\delta^{+}(Ky)}\phi(\overline{f})-\sum_{\overline{f}\in\delta^{-}(Ky)}\phi(\overline{f})
$$ 
is independent of the choice of the orientation of internal edges in $R$ by Lemma \ref{lem:first-flow}(1).\footnote{Note that a loop $\overline{f}$ of $Ky$ contributes $\phi(\overline{f})$ to both summations, and hence has net contribution zero to $\operatorname{div}(Ky)$.}

Next, we claim that for any quotient vertex $Ky\in V(R)$, it always holds that $$\operatorname{div}(Ky)=0.$$ 
If $Ky$ has no incident loop, then we can imagine that all incident edges are directed out of $Ky$. So we have
$$
\operatorname{div}(Ky)=S_1(y)+S_2(y)+S_3(y).
$$
By Lemma \ref{lem:first-flow}(3), we derive that
$\operatorname{div}(Ky)=0.$
Otherwise suppose that $Ky$ has a loop, say of label $i$. 
Then $m_i(y)=\tau(y)$ for some $\tau\in K$, and hence
\[
S_i(y)=S_i(\tau(y))=S_i(m_i(y))=-S_i(y).
\]
Since $p$ is odd, we have $S_i(y)=0$. Note that a loop makes zero net contribution to the divergence, while every incident non-loop edge of label $j$ contributes $S_j(y)$. Using Lemma \ref{lem:first-flow}(3), this again shows that $\operatorname{div}(Ky)=0.$

Summing the divergence over all vertices in $R$ gives $\sum_{Ky\in V(R)} \operatorname{div}(Ky)=0.$
By cancellation of the contributions from the two endpoints of each oriented edge in $R$, the only remaining term on the left-hand side is $S_1(x)$, contributed by the quotient edge $\overline{e}$. Therefore, we have
\(S_1(x)=\frac{\Delta_1(x)(x_2^2-x_3^2)}{x_2x_3}=0.\)
Since $R$ has no cusps, we have $x_2x_3\neq0$. Since $\overline{e}$ is a non-loop, we have $m_1(x)\neq x$, and thus $\Delta_1(x)\neq0$. It follows from the definition of $S_1$ that $x_2^2=x_3^2$, as desired.
\end{proof}

Recall the definition of $H_R(x)$ in \eqref{equ:H_R}.
We turn to consider the structure of $H_R(x)$ for a shore $R$ of a bridge $\overline{e}$ in $\overline{G}_{p,k}$, where $x$ is an endpoint of $e$ with $Kx\in V(R)$. It is worth to note the fact that 
$$
H_R(x)=H_R(\sigma(x))
$$
for any $\sigma\in K$. Indeed, for any $\sigma,\tau\in K$, if there is a path $P$ connecting $x$ and $\tau(x)$ in $G_{p,k}^{\circ}[\pi_k^{-1}(V(R))]$, then the path $\sigma(P)$ connects $\sigma(x)$ and $\sigma(\tau(x))=\tau(\sigma(x))$ in $G_{p,k}^{\circ}[\pi_k^{-1}(V(R))]$.
The next lemma characterizes $H_R(x)$ under certain conditions.

\begin{lemma}
\label{lem:diagonal-symmetry}
Let $\overline{e}$ be a quotient bridge with label 1. Suppose that there exists an endpoint $z$ of $e$ satisfying
\[
z_2^2=z_3^2\neq 0.
\]
Then, for each shore $R$ of $\overline{e}$, if $x$ is the endpoint of $e$ with $Kx\in V(R)$, we have
$$
H_R(x)=\{1\},\ \{1,\sigma_1\}\text{ or }K.
$$
\end{lemma}

\begin{proof}
Since $z_2^2=z_3^2$, we have $z_2=z_3$ or $z_2=-z_3$.
Since $H_R(x)=H_R(\sigma(x))$ for any $\sigma\in K$, we may assume that $z=(z_1,t,t)$ by replacing $e$ by $\sigma_2(e)$ if necessary. 
Since $e$ has label 1, the other endpoint of $e$ is $(t^2-z_1,t,t)$. In particular, if we write $x=(x_1,x_2,x_3)$, we must have
$$
x_2=x_3.
$$

We claim that $\sigma_2\in H_R(x)$ if and only if $\sigma_3\in H_R(x)$. By symmetry, it suffices to prove one direction. Suppose that $\sigma_2\in H_R(x)$. Then there exists a path $P$ in $G_{p,k}^{\circ}[\pi_k^{-1}(V(R))]$ from $x$ to $\sigma_2(x)$. Consider the bijection $\rho\colon X_{p,k}^{\circ}\to X_{p,k}^{\circ}$ defined by
\[
\rho(y_1,y_2,y_3)\coloneqq(y_1,y_3,y_2).
\]
It is straightforward to verify that
\begin{equation}
\label{eqn:rho_commutes}
\rho\circ m_1=m_1\circ\rho,\qquad \rho\circ m_2=m_3\circ\rho,\qquad \rho\circ m_3=m_2\circ\rho.
\end{equation}
Thus, $\rho$ also preserves edges.
Furthermore, since $x_2=x_3$, we have
$$
\rho(x)=x.
$$
Since $\rho\circ\sigma_2=\sigma_3\circ\rho$, we see that 
$$
\rho(\sigma_2(x))=\sigma_3(x).
$$
Therefore, we see that the path $\rho(P)$ connects $x=\rho(x)$ and $\sigma_3(x)=\rho(\sigma_2(x))$. Since $\rho(x)=x$ and $\rho$ preserves edges, it maps the connected component of $x$ in $G_{p,k}^{\circ}[\pi_k^{-1}(V(R))]$ to itself. Hence $\rho(P)$ is contained in $G_{p,k}^{\circ}[\pi_k^{-1}(V(R))]$, and so $\sigma_3\in H_R(x)$. This proves the claim.

Assuming that $\sigma_2,\sigma_3\in H_R(x)$, we show that $H_R(x)=K$. For any $i\in\{2,3\}$, there exists a path $P_i$ connecting $x$ and $\sigma_i(x)$ in $G_{p,k}^{\circ}[\pi_k^{-1}(V(R))]$. Since $K$ commutes with the Vieta involutions, $\sigma_2(P_3)$ is a path connecting $\sigma_2(x)$ and $\sigma_2\sigma_3(x)=\sigma_1(x)$ in $G_{p,k}^{\circ}[\pi_k^{-1}(V(R))]$. 
Hence, concatenating $P_2$ with $\sigma_2(P_3)$ yields a walk, and thus a path, connecting $x$ and $\sigma_1(x)$ in $G_{p,k}^{\circ}[\pi_k^{-1}(V(R))]$. It follows that $\sigma_1\in H_R(x)$, and therefore $H_R(x)=K$.

Combining the preceding two properties, we deduce that if either $\sigma_2$ or $\sigma_3$ belongs to $H_R(x)$, then $H_R(x)=K.$
Otherwise, $H_R(x)$ has to be either $\{1\}$ or $\{1,\sigma_1\}$. This finishes the proof.
\end{proof}

Let $R$ be a connected subgraph of $\overline{G}_{p,k}$. 
In the following lemma, we identify a special type of component of $G_{p,k}^{\circ}[\pi_k^{-1}(V(R))]$ arising when $Kx\in V(R)$ satisfies $H_R(x)=\{1\}$, as in Lemma~\ref{lem:diagonal-symmetry}.

\begin{lemma}
\label{lem:reachable-signs}
Let $R$ be a connected subgraph of $\overline{G}_{p,k}$. Suppose that $Kx\in V(R)$ with $H_R(x)=\{1\}$. Then the connected component $L$ of $G_{p,k}^{\circ}[\pi_k^{-1}(V(R))]$ containing $x$ contains exactly one vertex in $Ky$ for each quotient vertex $Ky\in V(R)$.
We call such a component $L$ a \emph{one-sheet lift} of $R$.
\end{lemma}
\begin{proof}
Since $R$ is connected, for each quotient vertex $Ky$ in $R$, we can choose a path from $Kx$ to $Ky$, and lift it to a path from $x$ to one vertex in $Ky$. Since $L$ is the connected component containing $x$, we see that $L$ contains at least one vertex in $Ky$. Suppose that $L$ contains two distinct points $\sigma_i(y)$ and $\sigma_j(y)$ in $Ky$, where $\sigma_i, \sigma_j\in K$ with $i\neq j$. Choose a path $P_1$ in $L$ from $x$ to $\sigma_i(y)$ and a path $P_2$ from $\sigma_j(y)$ to $x$. Then concatenating $P_1$ with $\sigma_i\sigma_j^{-1}(P_2)$ gives a path from $x$ to $\sigma_i\sigma_j^{-1}(x)$ in $G_{p,k}^{\circ}[\pi_k^{-1}(V(R))]$. This means that $\sigma_i\sigma_j^{-1}\in H_R(x)$, contradicting the assumption $H_R(x)=\{1\}$. Hence $L$ contains exactly one vertex in $Ky$. 
\end{proof}

Such a one-sheet lift is clearly an obstruction to our proof approach, as illustrated in Figure~\ref{fig:bridgelift}. 
To eliminate this obstruction, we will make use of graph flows once more, whose definition involves the denominators $x_1^2-4$ and $x_3^2-4$.
The next lemma will be used when dealing with the vertices $x\in X_{p,k}^{\circ}$ at which one of these denominators vanishes. Its proof is a refinement of the calculation in \cite[Proposition 7.1(a)]{de2024non}.

\begin{lemma}
\label{lem:pole-cycle}
Assume that $k\in\FF_p\setminus\{4\}$. Let $x=(x_1,x_2,x_3)\in X_{p,k}^{\circ}$ satisfy $x_3=2\varepsilon$ with $\varepsilon\in\{\pm1\}$. Then $x$ lies on a cycle of length $2p$ formed by edges with alternating labels 1 and 2, which contains a cusp. The analogous statement holds after permuting the coordinates.
\end{lemma}
\begin{proof}
By \cite[Proposition 7.1(a)]{de2024non}, the walk of length $2p$ formed by edges with alternating labels 1 and 2
\[
x=v_1\xline{m_1}w_1\xline{m_2}v_2\xline{m_1}w_2\xline{m_2}\cdots\xline{m_1}w_p\xline{m_2}v_{p+1}
\]
is a closed walk, that is, we have $v_1=v_{p+1}$. Next we prove that it is a cycle, by showing that these vertices $v_1,v_2,\ldots,v_p,w_1,w_2,\ldots,w_p$ are pairwise distinct in $X_{p,k}^{\circ}$. To see this, we define
$$
I(z)\coloneqq(z_1-\varepsilon z_2, z_2)
$$
for $z=(z_1,z_2,z_3)\in\{v_1,v_2,\ldots,v_p,w_1,w_2,\ldots,w_p\}$. Since $x_3=2\varepsilon$, the generalized Markoff equation becomes
\[
(x_1-\varepsilon x_2)^2=k-4.
\]
Set
\[
d\coloneqq x_1-\varepsilon x_2.
\]
Since $k\neq4$, we have $d\neq0$. Therefore, we have
\begin{equation}
\label{eqn:initial_invariant}
I(v_1)=(d,x_2).
\end{equation}
Note that for each $z=(z_1,z_2,z_3)\in\{v_1,v_2,\ldots,v_p,w_1,w_2,\ldots,w_p\}$, we have $z_3=2\varepsilon=x_3$. It also follows that if $I(z)=(d,z_2)$, then $z_1-\varepsilon z_2=d$ and
\[
I(m_1(z))=I\bigl((z_2z_3-z_1,z_2,z_3)\bigr)=(z_2z_3-z_1-\varepsilon z_2, z_2)=(\varepsilon z_2-z_1,z_2)=(-d,z_2), 
\]
and if $I(z)=(-d,z_2)$, then $z_1-\varepsilon z_2=-d$ and
\[
I(m_2(z))=I\bigl((z_1,z_1z_3-z_2,z_3)\bigr)=(z_1-\varepsilon z_1z_3+\varepsilon z_2,z_1z_3-z_2)=\bigl(\varepsilon z_2-z_1, (\varepsilon z_2-d)\cdot 2\varepsilon-z_2\bigr)=(d,z_2-2\varepsilon d).
\]
Applying these facts to (\ref{eqn:initial_invariant}) repeatedly, we see that
\[
I(v_i)=\bigl(d,x_2-2(i-1)\varepsilon d\bigr), \qquad \mbox{and} \qquad
I(w_i)=\bigl(-d,x_2-2(i-1)\varepsilon d\bigr).
\]
Because of $d\neq0$, we see that $I(v_1),I(v_2),\ldots,I(v_p),I(w_1),I(w_2),\ldots,I(w_p)$ are pairwise distinct. It follows that $v_1,v_2,\ldots,v_p,w_1,w_2,\ldots,w_p$ are pairwise distinct. Since the second coordinate of $I(z)$ records the second coordinate of $z$ and $d\neq0$, we see that the cycle contains a cusp (in fact, two). This completes the proof.
\end{proof}

For a vertex $x=(x_1,x_2,x_3)\in X_{p,k}^{\circ}$ satisfying $x_1^2\neq4$ and $x_3^2\neq4$, we define
\[
W_1(x)\coloneqq
-\frac{\Delta_1(x)\cdot x_3}{x_3^2-4},\quad
W_2(x)\coloneqq
\frac{2\Delta_2(x)\cdot (x_3^2-x_1^2)}
{(x_1^2-4)(x_3^2-4)},\quad \mbox{and} \quad
W_3(x)\coloneqq
\frac{\Delta_3(x)\cdot x_1}{x_1^2-4}.
\]

\begin{lemma}
\label{lem:twisted-flow}
The following statements hold whenever the functions under consideration are defined.
\begin{itemize}
\item[(1)] For every $i\in\{1,2,3\}$, we have 
$$
W_i(m_i(x))=-W_i(x).
$$
\item[(2)] We have 
$$
W_1(x)+W_2(x)+W_3(x)=0.
$$
\item[(3)] For every $i\in\{1,2,3\}$, we have
\[
W_i(\sigma_1(x))=-W_i(x),
\qquad
W_i(\sigma_2(x))=W_i(x),
\qquad \mbox{and} \qquad
W_i(\sigma_3(x))=-W_i(x).
\]
\end{itemize}
\end{lemma}
\begin{proof}
(1) This follows from $\Delta_i(m_i(x))=-\Delta_i(x)$.

(2) By clearing the denominators, we have
\[
W_1(x)+W_2(x)+W_3(x)=\frac{-\Delta_1(x)x_3(x_1^2-4)+2\Delta_2(x)(x_3^2-x_1^2)+\Delta_3(x)x_1(x_3^2-4)}{(x_1^2-4)(x_3^2-4)}.
\]
The numerator of the right-hand side vanishes by straightforward calculations. This proves (2).

(3) It follows from straightforward calculations, which we omit.
\end{proof}

We aim to use the functions $W_i$ to define yet another flow on the quotient graph when a one-sheet lift of a shore $R$ occurs (i.e., the obstruction discussed after Lemma~\ref{lem:reachable-signs}). 
To define such a flow, we need to assign a unique value to each quotient vertex $Kx\in V(R)$ using the functions $W_i$, in a manner analogous to the use of the $K$-invariant functions $S_i$ in the proof of Lemma~\ref{lem:diagonal-bridge}.
However, by Lemma~\ref{lem:twisted-flow}(3), the functions $W_i$ are not $K$-invariant.
The advantage here is that a specified one-sheet lift naturally provides a representative vertex in $X_{p,k}^{\circ}$ for each quotient vertex $Kx\in V(R)$, allowing us to proceed despite the lack of $K$-invariance.
We carry this out in the following lemma, which ultimately shows that the case $Kx\in V(R)$ with $H_R(x)=\{1\}$ in Lemma~\ref{lem:diagonal-symmetry} can never occur.

\begin{lemma}
\label{lem:nontrivial-signs}
Assume that $k\in\FF_p\setminus\{4\}$.
Let $\overline{e}$ be a quotient bridge with label $1$, and let $R$ be a shore of $\overline{e}$ containing no cusp.
Let $e$ be a lift of $\overline{e}$, and suppose that no cycle in $G_{p,k}^\circ$ contains $e$. 
Let $x$ be the endpoint of $e$ in $G_{p,k}^\circ[\pi_k^{-1}(V(R))]$. Then $H_R(x)\neq\{1\}$.
\end{lemma}

\begin{proof}
Suppose to the contrary that $H_R(x)=\{1\}$. By Lemma \ref{lem:reachable-signs}, there exists a one-sheet lift $L$ of $R$ containing $x$. Since $R$ contains no cusp, every vertex of $R$ is regular and by Lemma~\ref{lem:diagonal-bridge}, we have $x_2^2=x_3^2$.

First, it is easy to verify that $x_2^2=x_3^2\neq4$. Otherwise, if $x_3^2=4$, then there exists a cycle containing $e$ by Lemma \ref{lem:pole-cycle}, which contradicts our assumption.

Second, we claim that $W_1,W_2,W_3$ are defined at every vertex of the one-sheet lift $L$. 
In other words, for each vertex $y=(y_1,y_2,y_3)\in V(L)$, we have $y_1^2\neq4$ and $y_3^2\neq4$. Assume to the contrary that there exists $y\in V(L)$ with $y_1^2=4$. 
By Lemma~\ref{lem:pole-cycle}, $y$ lies on a cycle $C$ consisting of edges with alternating labels $2$ and $3$, which contains a cusp. 
Since $R$ contains no cusp and is a shore of the quotient bridge $\overline{e}$, there must be an edge of $C$ with label $2$ or $3$ that projects to $\overline{e}$.
However, by Lemma~\ref{lem:generalized-labels}, $\overline{e}$ has a unique label, which is $1$ by assumption, a contradiction.
Now suppose that there exists $y\in V(L)$ with $y_3^2=4$. 
Then by Lemma~\ref{lem:pole-cycle} again, $y$ lies on a cycle $C$ consisting of edges with alternating labels $1$ and $2$, which contains a cusp. 
Using the same conditions, we see that there must be a vertex $z=(z_1,z_2,z_3)\in V(C)$ which projects to $Kx$. This implies that $x_3^2=z_3^2$.
Since the cycle $C$ consists only of $m_1$- and $m_2$-edges, the third coordinate remains unchanged along $C$, and hence $y_3=z_3$.
Therefore, we have 
$$
x_3^2=z_3^2=y_3^2=4,
$$
which contradicts with the fact that $x_3^2\neq4$. The proof of this claim is complete.

Recall that $L$ contains exactly one vertex in each $K$-orbit $Ky$ for every quotient vertex $Ky\in V(R)$; we call this vertex the \emph{unique lift} of $Ky$ in $L$.
Choose an orientation for each quotient edge $\overline{f}\in E(R)\cup\{\overline{e}\}$ such that $\overline{e}$ is directed out of $R$. 
Since every vertex of $R$ is regular, Lemma~\ref{lem:generalized-labels} implies that each quotient edge $\overline{f}$ has a unique label, say $i$. We then assign it the value
$$
\eta(\overline{f})\coloneqq W_i(y),
$$
where $y$ is the unique lift in $L$ of the initial endpoint of $\overline{f}$. 
By Lemma \ref{lem:twisted-flow}(1), the divergence
$$
\operatorname{div}(Ky)=\sum_{\overline{f}\in\delta^{+}(Ky)}\eta(\overline{f})-\sum_{\overline{f}\in\delta^{-}(Ky)}\eta(\overline{f})
$$
is independent of the orientation chosen for the internal edges of $R$.
Using similar arguments in the proof of Lemma \ref{lem:diagonal-bridge} and applying Lemma \ref{lem:twisted-flow}, we can deduce that
$$
W_1(x)=0.
$$
Clearly $e$ is a non-loop (as otherwise $\overline{e}$ would be a loop), and thus we have $\Delta_1(x)\neq0$. Furthermore, because $R$ contains no cusps, we have $x_3\neq0$. It follows that
\[
W_1(x)=-\frac{\Delta_1(x)\cdot x_3}{x_3^2-4}\neq0.
\]
This final contradiction finishes the proof.
\end{proof}

Finally we are ready to prove Proposition~\ref{prop:generalized-bridge}.

\begin{proof}[\bf Proof of Proposition~\ref{prop:generalized-bridge}]
Let $A$ and $B$ be the two shores of $\overline{e}$. Write $e=uv$ with $Ku\in V(A)$ and $Kv\in V(B)$. Assume to the contrary that no cycle in $G_{p,k}^\circ$ contains $e$. Without loss of generality, we may assume that the label of $e$ is $1$. 

Now we claim that 
\[
\sigma_1\in H_A(u)\cap H_B(v).
\]
If $A$ contains a cusp, Lemma \ref{lem:cuspconnected} implies that $G_{p,k}^\circ[\pi_k^{-1}(V(A))]$ is connected. 
By definition, we have $\sigma_1\in H_A(u)$. 
If $A$ contains no cusps, by applying Lemma \ref{lem:diagonal-bridge}, we can derive that
\[
u_2^2=u_3^2.
\]
It then follows from Lemma \ref{lem:nontrivial-signs} that $H_A(u)\neq\{1\}$. Then Lemma \ref{lem:diagonal-symmetry} shows that $\sigma_1\in H_A(u)$. Hence, in both cases, we have $\sigma_1\in H_A(u)$. Similarly, we can also deduce that $\sigma_1\in H_B(v)$, proving the claim.

Therefore, there exists a path $P_A$ (resp. $P_B$) in $G_{p,k}^\circ[\pi_k^{-1}(V(A))]$ (resp. $G_{p,k}^\circ[\pi_k^{-1}(V(B))]$) connecting $\sigma_1(u)$ and $u$ (resp. $v$ and $\sigma_1(v)$). 
Together with the two crossing edges $e$ and $\sigma_1(e)$, these paths form a closed walk
\[
u \xline{e} v
\xline{P_B} \sigma_1(v)
\xline{\sigma_1(e)} \sigma_1(u)
\xline{P_A} u.
\]
Note that neither $P_A$ nor $P_B$ contains the edge $e$. Because the $K$ acts freely on $X_{p,k}^{\circ}$, we have $\sigma_1(e)\neq e$. Hence, the closed walk gives a cycle containing $e$ in $G_{p,k}^\circ$, which contradicts with the assumption that $e$ lies on no cycles. This finishes the proof of Proposition~\ref{prop:generalized-bridge}.
\end{proof}

\section{Proof of the main theorem}
\label{sec:main_thm}

Now we are ready to present the proof of our main theorem.

\begin{proof}[\bf Proof of Theorem~\ref{thm:generalized-main}]
First, Theorem \ref{thm:generalized-main}(1) follows directly from Lemma \ref{lem:small-components}. 
To prove Theorem \ref{thm:generalized-main}(2), consider any connected component $C$ of $G_{p,k}$ that contains no axial vertex. Then by Lemma \ref{lem:small-components}, $C$ is also a component in $G_{p,k}^{\circ}$ with
\(|V(C)|\geq3.\)

Let $e=\{x,y\}$ be an arbitrary non-loop edge of $C$. 
We first show that $e$ is contained in a cycle of $G_{p,k}^{\circ}$, distinguishing cases according to the type of $\overline{e}$.

\vspace{0.1cm}
\noindent\textbf{Case 1: $\overline{e}$ is a non-loop non-bridge edge in $\overline{G}_{p,k}$}.
\vspace{0.1cm}

Since $\overline{e}$ is a non-bridge edge in $\overline{G}_{p,k}$, there exists a path $\overline{P}$ from $Kx$ to $Ky$ in $\overline{G}_{p,k}-\overline{e}$. Lift it to a path $P$ in $G_{p,k}^\circ$ from $x$ to $\sigma(y)$ for some $\sigma\in K$. If $\sigma=1$, then $P\cup\{e\}$ contains a cycle through $e$. If $\sigma\neq1$, consider the following closed walk
\[
y \xline{e} x
\xline{P} \sigma(y)
\xline{\sigma(e)} \sigma(x)
\xline{\sigma(P)} \sigma^2(y)=y.
\]
Because $\overline{P}$ avoids $\overline{e}$, neither $P$ nor $\sigma(P)$ contains $e$. 
Moreover, $\sigma(e)\neq e$, as otherwise $\sigma$ would exchange the endpoints of $e$, making $\overline{e}$ a quotient loop, a contradiction.
Thus, the above closed walk traverses $e$ exactly once and consequently contains a cycle through $e$ in $G_{p,k}^{\circ}$.

\vspace{0.2cm}
\noindent\textbf{Case 2: $\overline{e}$ is a bridge in $\overline{G}_{p,k}$}.
\vspace{0.1cm}

By Proposition~\ref{prop:generalized-bridge}, the edge $e$ is contained in a cycle in $G_{p,k}^{\circ}$.

\vspace{0.2cm}
\noindent\textbf{Case 3: $\overline{e}$ is a loop in $\overline{G}_{p,k}$}.
\vspace{0.1cm}

Without loss of generality, we may assume that $e$ has label $1$. Since $e$ itself is not a loop, we have
\begin{equation}
\label{eqn:loop_condition}
m_1(x)=\sigma(x) 
\end{equation}
for some non-identity $\sigma\in K$. Write $x=(x_1,x_2,x_3)$. It follows from (\ref{eqn:loop_condition}) that
\[
\begin{array}{c|c}
\sigma & \text{Implications of \eqref{eqn:loop_condition}}\\
\hline
\sigma_1 & x_1=x_2=x_3=0,\\
\sigma_2 & x_3=0,\\
\sigma_3 & x_2=0.
\end{array}
\]
In all three cases, $x$ is either an axial vertex or a cusp.
Since the connected component $C$ contains no axial vertices, it suffices to consider the case when $x$ is a cusp. In this case, Lemma~\ref{lem:cusp4cycle} implies that the edge $e=\{x,\sigma(x)\}$ lies on a $4$-cycle in $G_{p,k}^{\circ}$, as desired. 

We have shown in all of the above three cases that every non-loop edge of the component $C$ is contained in a cycle of $G_{p,k}^{\circ}$. 
It remains to show that the component $C$ is $2$-connected. Suppose to the contrary that $C-\{v\}$ has at least two connected components for some vertex $v$. Each such component must be joined to $v$ by at least two edges; otherwise, its unique attaching edge would be a bridge of $C$, contradicting the statement just proved.
Therefore $v$ has at least four distinct non-loop neighbors in $C$. This is impossible because the three Vieta involutions give every vertex at most three distinct non-loop neighbors. Since $|V(C)|\geq3$, the component $C$ is $2$-connected. This completes the proof.
\end{proof}

\section*{Acknowledgements}
The authors are grateful to Joseph Silverman for his insightful comments and suggestions on an earlier version, which led us to investigate the generalized Markoff graphs.
J.M. is supported by the National Key Research and Development Program of China, grant 2023YFA1010201, and by the National Natural Science Foundation of China, grant 12125106.
Z.Y. is supported by the China Postdoctoral Science Foundation - Anhui Joint Support Program under Grant Number 2025T001AH and by Mozi Outstanding Young Talent Special Subsidy, provided by University of Science and Technology of China.

\section*{Declaration on the Use of AI}

During the preparation of an initial version of this work, which focused on the Markoff graphs $G_p$, the authors used AI systems to assist in generating candidate proof strategies, particularly in identifying suitable functions for constructing the flows used in the proof. The authors subsequently found that the method extends to the generalized Markoff graphs $G_{p,k}$. All AI-generated suggestions were independently verified and refined by the authors, who take full responsibility for the correctness of the paper.

\bibliographystyle{plain}

\bigskip

\end{document}